\documentclass[10pt]{amsart}
\usepackage{amsmath}
\usepackage{epsfig}
\usepackage{graphics,graphicx}
\usepackage{amsfonts}
\usepackage{caption}
\usepackage{verbatim}
\usepackage{amssymb}
\usepackage{latexsym}
\usepackage{amsmath, amsthm, amssymb}

\newtheorem{corollary}{Corollary}[section]

\newtheorem{lemma}[corollary]{Lemma}
\newtheorem{proposition}[corollary]{Proposition}
\newtheorem{remark}[corollary]{Remark}
\newtheorem{theorem}[corollary]{Theorem}
\newcommand{\mylabel}[1]{\label{#1}
            \ifx\undefined\stillediting
            \else \fbox{$#1$}\fi }
\newcommand{\BE}{\begin{equation}}

\newcommand{\EEQ}{\end{equation}}
\newcommand{\rfb}[1]{\mbox{\rm
   (\ref{#1})}\ifx\undefined\stillediting\else:\fbox{$#1$}\fi}

\newfont{\Blackboard}{msbm10 scaled 1200}
\newcommand{\bl}[1]{\mbox{\Blackboard #1}}
\newfont{\roma}{cmr10 scaled 1200}

\def\CC{\rm \hbox{C\kern-.56em\raise.4ex
         \hbox{$\scriptscriptstyle |$}\kern+0.5 em }}

\newcommand{\nline}  {{\bl N}}
\newcommand{\rline}  {{\bl R}}

\newcommand{\mm}    {{\hbox{\hskip 0.5pt}}}

\newcommand{\bluff} {{\hbox{\raise 15pt \hbox{\mm}}}}
\providecommand{\abs}[1]{\lvert#1\rvert}

\makeatletter
\def\section{\@startsection {section}{1}{\z@}{-3.5ex plus -1ex minus
    -.2ex}{2.3ex plus .2ex}{\large\bf}}
\makeatother
\def\be{\begin{equation}}
\def\ee{\end{equation}}

\def\ds{\displaystyle}
\newcommand{\norm}[2]{\|#1 \| _{#2} }

\begin{document}

\thispagestyle{empty}
\title[Stabilization and controllability of the KdV on a star-shaped network]{Feedback stabilization and boundary controllability  of the  Korteweg-de Vries equation on a star-shaped network}
\date\today
\author{Ka\"{i}s Ammari}
\address{UR Analysis and Control of PDEs, UR 13ES64, Department of Mathematics, Faculty of Sciences of Monastir, University of Monastir, Tunisia and Laboratoire de Math\'ematiques, Universit\'e de Versailles Saint-Quentin en Yvelines, 78035 Versailles, France and Universit\'e Paris-Saclay, France}
\email{kais.ammari@fsm.rnu.tn}
\author{Emmanuelle Crepeau}
\address{Laboratoire de Math\'ematiques, Universit\'e de Versailles Saint-Quentin en Yvelines, 78035 Versailles, France}
\email{emmanuelle.crepeau@uvsq.fr}

\begin{abstract}
We propose a model using  the Korteweg-de Vries $(KdV)$ equation on a finite star-shaped network. We first prove the well-posedness of the system and give some regularity results. Then we prove that the energy of the solutions of the dissipative system decays exponentially to zero when the time tends to infinity. Lastly we show an exact boundary controllability result.
 
\medskip

\noindent
{\bf R\'esum\'e.}
On propose dans cet article un mod\`ele de l'\'equation de Korteweg-de Vries $(KdV)$ sur un r\'eseau sous forme d'une \'etoile. On prouve que le probl\`eme est bien pos\'e et on \'etablit quelques propri\'et\'es de r\'egularit\'e. De plus, on montre que l'\'energie du syst\`eme d\'ecroit d'une mani\`ere exponentielle vers $0$ quand le temps tend vers l'infini. A la fin, on d\'eduit un r\'esultat de contr\^olabilit\'e fronti\`ere du syst\`eme associ\'e. 

\end{abstract}


\maketitle

\tableofcontents

 
\section{Introduction} \label{secintro}

\medskip

In the last few years various physical models of multi-link flexible structures consisting of finitely many interconnected flexible elements such as strings, beams, plates, shells have been mathematically studied. For details about some physical motivation for the models, see 
\cite{dagerzuazua, ammari4, amjel, ammarinicaise} and the references therein.

 In \cite{Crepeau-Sorine}, the Korteweg-de Vries equation (KdV) is designed for modeling  the pressure in an arterial compartment. Indeed, the Korteweg-de Vries equation models usually long waves in a channel of relatively shallow depth.  Thus we propose a new model using this nonlinear dispersive partial differential equation on a network to be used to model the pressure on the arterial tree.

Numerous papers on the stability or the exact controllability of the KdV equation on a finite length interval have already been studied, see for example \cite{Perla, Pazoto} for the stability and \cite{Rosier, Coron_Crepeau, Cerpa, Cerpa_Crepeau} for the control problem. In \cite{Cerpa}, a tutorial of both problems is presented. 

To our knowledge, there is no work about the KdV equation on a star-network but we can cite the article \cite{Crepeau} where the  controllability of the KdV equation on a compartment with nodes is presented.

\medskip

Now, let us first introduce some notations and definitions which will be used throughout the rest of the paper, in particular some which are linked to the notion of $C^{\nu }$- networks, $\nu \in \nline$ (as introduced in \cite{dagerzuazua}). 

\medskip

Let $\Gamma$ be a connected topological graph embedded in $\rline$, with $N$ edges ($N \in \nline^{*}$).  
Let $K=\{k_{j}\, :\, 1\leq j\leq N\}$ be the set of the edges of $\Gamma$. Each edge $k_{j}$ is a Jordan curve in $\rline$ and is assumed to be parametrized by its arc length $x_{j}$ such that
the parametrization $\pi _{j}\, :\, [0,\ell_j]\rightarrow k_{j}\, :\, x_{j}\mapsto \pi _{j}(x_{j})$ is $\nu$-times differentiable, i.e. $\pi _{j}\in C^{\nu }([0,\ell_j],\rline)$ for all $1\leq j\leq N$. 
The $C^{\nu}$- network $\mathcal{T}$  associated with $\Gamma$ is then defined as the union $${\mathcal T}=\bigcup _{j=1}^{N}k_{j}.$$

We define by $L:= \ds \sup_{j=1,..,N}{\ell_j}$, the maximal length of the network.

We study here the stabilization problem and the controllability one of a KdV system on a star-shaped network as in the following figure $1$ for $N =3$. 

\newpage

\begin{center} \label{fig}
\includegraphics[scale=1.40]{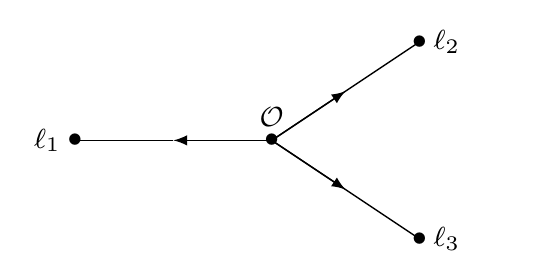}
\captionof{figure}{Star-Shaped Network for $N =3$}
\end{center}

More precisely, we study a system which is in connection with the mathematical modeling of the human cardiovascular system. For each edge $k_{j}$, the scalar function $u_j(t,x)$ for $x \in (0,\ell_j)$ and $t > 0$
contains the information on the displacement of the wave at location $x$ and time $t$, $1 \leq j \leq N$. 

\smallskip

\setcounter{equation}{0}
We consider the evolution problems  $(KdV)$ and $(LKdV)$  described by the following systems: 

\begin{equation*}
\leqno(KdV) 
\left \{
\begin{array}{ll}
(\partial_t u_{j} + \partial_x u_{j} + u_j \partial_x u_j + \partial^3_x u_j)(t,x)=0,&\, \forall \, x \in(0,\ell_j),\, t\in(0,+\infty),\, j = 1,...,N, \\
u_{j}(t,0)=u_{k}(t,0),&\, \forall \, j,k = 1,...,N, \, t > 0, \\
\ds \sum_{j=1}^N \partial^2_{x} u_j (t,0) = - \alpha \, u_1(t,0) - \frac{N}{3} u^2_1(t,0),& \, \forall \, t > 0, \\
u_j(t,\ell_j) = \partial_x u_j (t,\ell_j) = 0,& \, \forall \, t > 0, \, j= 1,...,N, \\
u_j(0,x)=u_j^0(x),& \, \forall \, x \in (0,\ell_j),\,   j=1,...,N,
\end{array}
\right.
\end{equation*}

and 

\begin{equation*}
\leqno(LKdV) 
\left \{
\begin{array}{ll}
(\partial_t u_{j} + \partial_x u_{j} + \partial^3_x u_j)(t,x)=0,&\, \forall \, x \in(0,\ell_j),\, t\in(0,\infty),\, j = 1,...,N, \\
u_{j}(t,0)=u_{k}(t,0),&\, \forall \, j,k = 1,...,N, \, t > 0, \\
\ds \sum_{j=1}^N \partial^2_{x} u_j (t,0) = - \alpha \, u_1(t,0),& \, \forall \, t > 0, \\
u_j(t,\ell_j) = \partial_x u_j (t,\ell_j) = 0,& \, \forall \, t > 0, \, j= 1,...,N, \\
u_j(0,x)=u_j^0(x),& \, \forall \, x \in (0,\ell_j),\,   j=1,...,N,
\end{array}
\right.
\end{equation*}

where $\alpha > \frac{N}{2}$.

We define the natural energy $E(t)$ of a solution $\underline{u} = (u_1,...,u_N)$ of $(KdV)$ or  $(LKdV)$ system by
\be 
\label{energy1}
E(t)=\frac{1}{2} \ds \sum_{j=1}^{N} \int_{0}^{\ell_j} |u_{j}(t,x)|^2  \, dx.
\ee

We can easily check that every sufficiently smooth solution of $(KdV)$ satisfies the following dissipation law 
\begin{equation}\label{dissipae1}
E^\prime(t) = - \ds \left(\alpha - \frac{N}{2}\right) \, \ds \bigl|u_{1}(t,0)\bigr|^2 - \frac{1}{2} \, \sum_{j=1}^N \left| \partial_x u_j(t,0)\right|^2 \leq 0, \, 
\end{equation}
and therefore, the energy is a nonincreasing function of the time variable $t$.

\medskip

This paper is organized as follows:
In Section \ref{well-posedness}, we give the proper functional setting for both systems $(LKdV)$ and $(KdV)$ and prove that those systems are well-posed. We also give some regularity results. 
In Section \ref{Stability}, we prove our main results, namely the stabilization problem of the systems given by $(LKdV)$ and $(KdV)$. For doing this, we derive first an observability inequality for the linear system and then we apply a fixed point theorem for the non-linear one.
In the last Section \ref{Control} we prove that the observability inequality also gives the controllability result in the case where the network is non critical.

\section{Well-posedness and regularity results} \label{well-posedness}

In order to study both systems on the network,  we need a proper functional setting. 
We define the following spaces:  
$$
H^s_r (0,\ell_j) = \left\{v \in H^{s} (0,\ell_j), \, \left(\frac{d}{dx}\right)^{i-1}  v (\ell_j) = 0, 1\leq i \leq s\right\}, \, s=1,2,
$$
$$
{\mathbb H}^s_e (\mathcal{T}) = \left\{\underline{u}=(u_1,...,u_N) \in \ds \prod_{j=1}^N H_r^s(0,\ell_j), \\ 
u_j(0) = u_k (0), \, \forall \, j,k =1, \ldots ,N \right\}, s=1,2,
$$
and
$$ 
{\mathbb L}^2(\mathcal{T}) = \ds \prod_{j=1}^{N} L^2(0,\ell_j),
$$
equipped with the inner product
\begin{equation}\label{ipV}
(\underline{u},\underline{v})_{{\mathbb L}^2 (\mathcal{T})} = 
\ds \sum_{j=1}^{N} \int_0^{\ell_j} u_j \overline{v}_j  \, 
dx, \, \forall \, \underline{u}, \underline{v} \in {\mathbb L}^2 (\mathcal{T}).
\end{equation}
We also define the following space $\mathbb{B}:=C([0,T],\mathbb{L}^2(\mathcal T))\cap L^2(0,T; \mathbb H^1_e(\mathcal T))$ endowed with the norm $$\norm{\underline u }{\mathbb{B}}:=\norm{\underline u}{C([0,T],\mathbb{L}^2(\mathcal T))}+\norm{\underline u}{L^2(0,T; \mathbb H^1_e(\mathcal T))}=\max_{t\in [0,T]}\norm{\underline u(t,.)}{L^2(\mathcal T)}+\left(\int_0^T\norm{\underline u(t,.)}{\mathbb H^1_e(\mathcal T)}^2dt\right)^{1/2}.$$

\subsection{Well-posedness of the $(LKdV)$ system and regularity results}

The system $(LKdV)$ can be rewritten as the first order evolution equation
\begin{equation} \left\{
\begin{array}{l}
U^\prime =\mathcal{A} U,\\
U(0)=( \underline{u}^{0}) ^T= U_0,
\end{array}\right.\label{pbfirstorder}\end{equation}
where $U$ is the vector $ \underline{u}^T$ and the operator $\mathcal{A} : {\mathcal D}({\mathcal A}) \subset {\mathbb L}^2 (\mathcal{T}) \rightarrow {\mathbb L}^2 (\mathcal{T})$ is defined by 
$$
\mathcal{A} \underline{u}^T := - \, \left(D_{\mathcal{T}} + D^3_{\mathcal{T}} \right) \underline{u}^T,
$$
$$
\forall \, \underline{u} 
\in {\mathcal D}({\mathcal A}) = \left\{\underline{u} = (u_1,\ldots,u_N) \in {\mathbb H}^2_e (\mathcal{T}) \cap \prod_{j=1}^N H^3(0,\ell_j), \, 
\ds \sum_{j=1}^N \ds \frac{d^2u_j}{dx^2} (0) = - \alpha \, u_1(0) \right\},
$$
and
$$ 
D_{\mathcal{T}} \underline{u}^T:= \left( 
\begin{array}{c}
\partial_x u_1 \\ \ldots \\ \partial_x u_N \end{array} \right), \, \forall \, 
\underline{u} \in \prod_{j=1}^N H^1(0,\ell_j).
$$

Now we can prove, according to the linear semi-group theory (see \cite{Pazy}), the well-posedness of system $(LKdV)$ and that the solution satisfies the dissipation law (\ref{dissipae1}).

\begin{proposition}\label{3exist1} 
For an initial datum $U_{0}\in \mathbb{L}^2(\mathcal{T})$, there exists a unique solution $U(t) := e^{t{\mathcal A}} U_0 \in C([0,\,+\infty),\, \mathbb{L}^2(\mathcal{T}))$ to problem (\ref{pbfirstorder}).
Moreover, the solution $\underline{u}$ satisfies \rfb{dissipae1}.
Therefore the energy is decreasing.
\end{proposition}

\begin{proof}
The operator $\mathcal A$ is clearly closed. Let $\underline{u}\in  {\mathcal D}({\mathcal A})$, then by using some integration by parts, we get,
\begin{equation*}
\begin{split}
(\underline u^T, \mathcal A \underline u^T)&=\sum_{j=1}^N\int_0^{\ell_j}u_j(-\overline u_{jx}-\overline u_{jxxx})dx\\
&= \left(\frac{N}{2}-\alpha \right) \abs{u_1(0)}^2-\frac{1}{2}\sum_{j=1}^N \abs{\partial_x u_{j}(0)}^2\leq 0.
\end{split}
\end{equation*}
Thus $\mathcal A$ is dissipative. 

\noindent The adjoint operator of $\mathcal A$ is defined by $\mathcal A^*\underline v^T:=  \, \left(D_{\mathcal{T}} + D^3_{\mathcal{T}} \right) \underline{v}^T,$ 
with
$$
{\mathcal D}({\mathcal A^*}) = \left\{\begin{array}{l}\underline{v} = (v_1,\ldots,v_N) \in  {\mathbb H}^1_e (\mathcal{T}) \ds \cap\prod_{j=1}^N H^3(0,\ell_j), \\
\ds\frac{d v_j}{dx}(0)=0,\, 
\ds \sum_{j=1}^N \ds \frac{d^2v_j}{dx^2} (0) = (\alpha-N) \, v_1(0) \end{array}\right\}.
$$
In the same manner, we obtain,
\begin{equation*}
(\underline v^T,\mathcal A^*\underline v^T)= \left(\frac{N}{2}-\alpha \right) \abs{v_1(0)}^2-\frac{1}{2}\sum_{j=1}^N \abs{\partial_x v_{j}(\ell_j)}^2\leq 0,
\end{equation*}
hence $\mathcal A^*$ is also dissipative and then $\mathcal A$ generates a strongly semi-group of contractions on ${\mathbb L}^2 (\mathcal{T})$. We denote by $S$ this semi-group.

\end{proof}

We also need some regularity results for the solution of the linear equation with some extra boundary conditions, 
\begin{equation}\label{KdV-g1}
\left \{
\begin{array}{ll}
(\partial_t u_{j} + \partial_x u_{j} + \partial^3_x u_j)(t,x)=0,&\, \forall \, x \in(0,\ell_j),\, t\in(0,\infty),\, j = 1,\ldots,N, \\
u_{j}(t,0)=u_{k}(t,0),&\, \forall \, j,k = 1,\ldots,N, \, t > 0, \\
\ds \sum_{j=1}^N \partial^2_{x} u_j (t,0) = - \alpha \, u_1(t,0)+g(t),& \, \forall  \, t > 0, \\
u_j(t,\ell_j) = \partial_x u_j (t,\ell_j) = 0,& \, \forall \, t > 0, \, j= 1,\ldots,N, \\
u_j(0,x)=u_j^0(x), \, \forall \, x \in (0,\ell_j),&\,   j=1,\ldots,N.
\end{array}
\right.
\end{equation}
\begin{proposition}
Let $(\underline u^0,g)\in \mathcal D(\mathcal A)\times C^2_0([0,T])$, where $\ds C^2_0([0,T]):=\{\varphi\in C^2([0,T]), \varphi(0)=0\}$. Then there exists a unique solution $\underline u\in C([0,T],\mathcal D(\mathcal A))\cap C^1([0,T]),\mathbb{L}^2(\mathcal T))$ of \eqref{KdV-g1}.

\end{proposition}
\begin{proof}
We first define the functions $\ds \phi_j(x):=\frac{(x-\ell_j)^2}{\ell_j^2\left(2\ds\sum_{i=1}^N \ell_i^{-2}+\alpha \right)}$. Thus $\phi_j\in C^\infty([0,\ell_j])$ and satisfies,
$$\left\{\begin{array}{ll}
\phi_j(\ell_j)=\phi_j'(\ell_j)=0,&\forall j=1,\ldots,N,\\
\phi_j(0)=\frac{1}{2\ds\sum_{i=1}^N\ell_i^{-2}+\alpha}=\phi_k(0),\, &\forall j,k=1,\ldots,N,\\
\ds \sum_{j=1}^N \phi_j''(0)=1-\alpha \phi_1(0).&
\end{array}\right.$$

We define $\underline z:=\underline u-g\underline \phi$, then $\underline z$ satisfies the system:
\begin{equation}\label{KdV-z}
\left \{
\begin{array}{ll}
(\partial_t z_{j} + \partial_x z_{j} + \partial^3_x z_j)(t,x)=-\phi_j(x) g'(t)-&(\phi_j'+\phi_j''')(x)g(t),\\
 &\forall \, x \in(0,\ell_j),\, t>0),\, j = 1,\ldots,N, \\
z_{j}(t,0)=z_{k}(t,0), \,& \forall \, j,k = 1,\ldots,N, \, t > 0, \\
\ds \sum_{j=1}^N \partial^2_{x} z_j (t,0) = - \alpha \, z_1(t,0), \, &\forall \,  t > 0, \\
z_j(t,\ell_j) = \partial_x z_j (t,\ell_j) = 0, \, \forall \,& t > 0, \, j= 1,\ldots,N, \\
z_j(0,x)=u_j^0(x), \, \forall \, x \in (0,\ell_j),\, &  j=1,\ldots,N.
\end{array}
\right.
\end{equation}

Thus, as $-\underline \phi g'-(\underline \phi'+\underline \phi''')g\in C^1([0,T],\mathbb{L}^2(\mathcal T))$, we deduce from Proposition \ref{3exist1} and classical results on semi-group theory, that system \eqref{KdV-z} admits a unique classical solution $\underline z\in C([0,T],\mathcal D(\mathcal A))\cap C^1([0,T]),\mathbb{L}^2(\mathcal T))$. Hence we can easily prove that problem \eqref{KdV-g1} admits a unique classical solution $\underline u\in C([0,T],\mathcal D(\mathcal A))\cap C^1([0,T]),\mathbb{L}^2(\mathcal T))$.

\end{proof}

Now, we study the same system but with less regular data.
\begin{proposition}\label{well_posed_LKdV}
Let $(\underline u^0,g)\in \mathbb{L}^2(\mathcal T)\times L^2(0,T)$, then there exists a unique mild solution of \eqref{KdV-g}, $\underline{u} \in \mathbb{B}$. Furthermore $\underline{u}(.,0)$ and $\partial_x \underline u(.,0)$ belong to $L^2(0,T)$ and we have the following  estimates,
\begin{equation}\label{L2H1etCL2}
\norm{\underline u}{\mathbb{B}}^2 
\leq C(T,N,L,\alpha) (\norm{\underline u^0}{\mathbb{L}^2(\mathcal T)}^2+\norm{g}{L^2(0,T)}^2),
\end{equation}
\begin{equation}\label{trace-u-u_x}
\begin{split}
\norm{u_1(.,0)}{L^2(0,T)}^2&\leq \frac{1}{\alpha-\frac{N}{2}}\norm{\underline u^0}{\mathbb{L}^2(\mathcal T)}^2+\frac{1}{(\alpha-\frac{N}{2})^2}\norm{g}{L^2(0,T)}^2,\\
\norm{\partial_x \underline u(.,0)}{L^2(0,T)}^2&\leq \norm{\underline u^0}{\mathbb{L}^2(\mathcal T)}^2+\frac{1}{\alpha-\frac{N}{2}}\norm{g}{L^2(0,T)}^2,
\end{split}
\end{equation}
\begin{equation}\label{u^0_bis}
\norm{\underline u^0}{\mathbb{L}^2(\mathcal{T})}^2\leq\frac{1}{T}\norm{\underline u}{L^2(0,T,\mathbb{L}^2(\mathcal T))}^2+3(\alpha-\frac{N}{2})\norm{ u_1(.,0)}{L^2(0,T)}^2+\norm{\partial_x\underline u(.,0)}{L^2(0,T)}^2+\frac{1}{\alpha-\frac{N}{2}}\norm{g}{L^2(0,T)}^2.
\end{equation}
\end{proposition}

\begin{proof}
The proof of this result is obtained by a density argument and the multiplier method. We first suppose that $(\underline u^0,g)\in \mathcal D(\mathcal A)\times C^2_0([0,T])$ and thus the solution of \eqref{KdV-g1} satisfies  $\underline u\in  C([0,T],\mathcal D(\mathcal A))\cap C^1([0,T]),\mathbb{L}^2(\mathcal T))$.

Let $\underline q=(q_j)_{j=1,\ldots,N} \in \ds \prod_{j=1}^N C^\infty([0,T]\times [0,\ell_j]; \mathbb{R})$.Then by multiplying $(LKdV)$ by $q_j \bar u_j$, integrating on $[0,s]\times [0,\ell_j]$ with $s\in [0,T]$ and using some integrations by parts we get the following equation,

\begin{equation}\label{multiplier}
\begin{split}
&\sum_{j=1}^N\int_0^{\ell_j}[q_j\abs{u_j}^2]_0^s dx-\sum_{j=1}^N\int_0^s\int_0^{\ell_j}(\partial_t q_{j}+\partial_x q_{j}+\partial^3_x q_{j})\abs{u_j}^2dxdt+3\sum_{j=1}^N\int_0^s\int_0^{\ell_j}\partial_x q_{j}\abs{\partial_x u_j}^2 \, dx \, dt \\
&=\sum_{j=1}^N\int_0^s \bigg((q_j+\partial^2_x q_j)\abs{u_j}^2+2q_ju_j\partial^2_x \bar{u}_j-2\partial_xq_ju_j\partial_x \bar{u}_j-q_j\abs{\partial_x u_j}^2 \bigg)(t,0) \, dt.
\end{split}
\end{equation}  

\begin{enumerate}
\item Taking first $\underline q= \underline{1}$, then \eqref{multiplier} becomes,
\begin{multline*}
\sum_{j=1}^N \int_0^{\ell_j} \abs{u_j (s,x)}^2 \, dx + \int_0^s \sum_{j=1}^N \abs{\partial_x u_j(t,0)}^2 \, dt + (2\alpha-N) \int_0^s\abs{u_1(t,0)}^2 \, dt=
\\
\sum_{j=1}^N\int_0^{\ell_j}\abs{u_j (0,x)}^2dx+2\int_0^s \bar u_1(t,0)g(t)dt \\
\leq \sum_{j=1}^N\int_0^{\ell_j}\abs{u_j (0,x)}^2\,dx+ \left(\alpha-\frac{N}{2} \right) \int_0^s\abs{u_1(t,0)}^2 \, dt+\frac{1}{\alpha-\frac{N}{2}}\int_0^s g^2(t)\,dt .
\end{multline*}
So we have,
 \begin{equation}\label{inequ_1}
\sum_{j=1}^N \int_0^{\ell_j} \abs{u_j (s,x)}^2 \, dx + \int_0^s \sum_{j=1}^N \abs{\partial_x u_j(t,0)}^2 \, dt + (\alpha-\frac{N}{2}) \int_0^s\abs{u_1(t,0)}^2 \, dt\leq 
\end{equation}
$$
\sum_{j=1}^N\int_0^{\ell_j}\abs{u_j (0,x)}^2\,dx+\frac{1}{\alpha-\frac{N}{2}}\int_0^T g^2(t)\,dt.
$$

\noindent Thus, $\underline u(s,.)\in \mathbb L^2(\mathcal T)$ and we have the estimate,\begin{equation}\label{contraction}
\max_{s\in [0,T]}\norm{\underline u(s,.)}{\mathbb L^2(\mathcal T)}^2\leq \norm{\underline u^0}{\mathbb L^2(\mathcal T)}^2+\frac{1}{\alpha-\frac{N}{2}}\norm{g}{L^2(0,T)}^2.
\end{equation}
\noindent Taking $s=T$ in inequality \eqref{inequ_1} gives that  $u_1(.,0)\in L^2(0,T)$ and $\partial_x u_j(.,0)\in L^2(0,T)$, for all $j=1,\ldots,N$ and we have the estimates,
\begin{equation}\label{trace}
\begin{split}
\norm{u_1(.,0)}{L^2(0,T)}^2&\leq \frac{1}{\alpha-\frac{N}{2}}\norm{\underline u^0}{\mathbb{L}^2(\mathcal T)}^2+\frac{1}{(\alpha-\frac{N}{2})^2}\norm{g}{L^2(0,T)}^2,\\
\norm{\partial_x \underline u(.,0)}{L^2(0,T)}^2&\leq \norm{\underline u^0}{\mathbb{L}^2(\mathcal T)}^2+\frac{1}{\alpha-\frac{N}{2}}\norm{g}{L^2(0,T)}^2.
\end{split}
\end{equation}

\item Secondly, we take $q_j(t,x)=x$ for $j=1,\ldots,N$ and  $s=T$ then equation \eqref{multiplier} gives us,
\begin{equation*}
\begin{split}
&\sum_{j=1}^N\int_0^{\ell_j}x[\abs{u_j}^2]_0^T dx-\sum_{j=1}^N\int_0^T\int_0^{\ell_j}\abs{u_j}^2dxdt+3\sum_{j=1}^N\int_0^T\int_0^{\ell_j}\abs{\partial_xu_j}^2dxdt\\
&=\sum_{j=1}^N\int_0^T\bigg(-2u_j\partial_x \bar{u}_j \bigg)(t,0)dt.
\end{split}
\end{equation*}
Then we have,

\begin{multline*}
\sum_{j=1}^N\int_0^T\int_0^{\ell_j}\abs{\partial_xu_j}^2dxdt\leq \frac{L}{3}\sum_{j=1}^N\int_0^{\ell_j}\abs{u_j(0,x)}^2dx\\+\frac{1}{3}\sum_{j=1}^N\int_0^T\int_0^{\ell_j}\abs{u_j}^2 \, dx \, dt + \frac{1}{3} \sum_{j=1}^N\int_0^T\bigg(-2u_j\partial_x \bar{u}_j \bigg)(t,0) \, dt.
\end{multline*}
Using estimates \eqref{contraction} and \eqref{trace}, we can deduce the following estimate,
\begin{equation}\label{H1_estimate}
\norm{\partial_x \underline u}{L^2([0,T],\mathbb{L}^2(\mathcal T))}^2\leq C(T,L,N,\alpha)\left(\norm{\underline u^0}{\mathbb{L}^2(\mathcal T)}^2+\norm{g}{L^2(0,T)}^2\right).
\end{equation}

\item Lastly, we choose $q_j(t,x)= T-t$ for $j=1,\ldots,N$ and $s=T$ then we obtain the equation,

\begin{equation*}\begin{split}
T\sum_{j=1}^N\int_0^{\ell_j}\abs{u_j(0,x)}^2dx&=\sum_{j=1}^N\int_0^T\int_0^{\ell_j}\abs{u_j}^2 \, dx \, dt + (2\alpha-N) \int_0^T(T-t)\abs{u_1(t,0)}^2dt  \\
&+\sum_{j=1}^N\int_0^T(T-t)\abs{\partial_x u_j(t,0)}^2dt-2\int_0^T (T-t)\bar u_1(t,0)g(t)dt,
\end{split}\end{equation*}
and then we easily get
\begin{equation*}
\norm{\underline u^0}{\mathbb{L}^2(\mathcal{T})}^2\leq\frac{1}{T}\norm{\underline u}{L^2(0,T,\mathbb{L}^2(\mathcal T))}^2+3(\alpha-\frac{N}{2})\norm{ u_1(.,0)}{L^2(0,T)}^2+\norm{\partial_x\underline u(.,0)}{L^2(0,T)}^2+\frac{1}{\alpha-\frac{N}{2}}\norm{g}{L^2(0,T)}^2.
\end{equation*}
\end{enumerate}

By the density of $\mathcal D(\mathcal A)$ in $\mathbb{L}^2(\mathcal T)$ and of $C^2_0([0,T])$ in $L^2(0,T)$, and by using inequalities \eqref{contraction}, \eqref{trace} and \eqref{H1_estimate}, we get the desired result.
\end{proof}

Before proving the well-posedness of $(KdV)$ we need also a result of regularity for the linear system with a source term.

\begin{proposition}\label{well_posed_source_term}
Let $(\underline u_0,\underline f,g)\in \mathbb{L}^2(\mathcal T)\times L^1(0,T,\mathbb{L}^2(\mathcal T))\times L^2(0,T)$, then there exists a unique mild solution  $\underline{u} \in \mathbb B$ of
\begin{equation}\label{KdV-source}
\left \{
\begin{array}{ll}
(\partial_t u_{j} + \partial_x u_{j} + \partial^3_x u_j)(t,x)=f_j,&\, \forall \, x \in(0,\ell_j),\, t>0,\, j = 1,\ldots,N, \\
u_{j}(t,0)=u_{k}(t,0),&\, \forall \, j,k = 1,\ldots,N, \, t > 0, \\
\ds \sum_{j=1}^N \partial^2_{x} u_j (t,0) = - \alpha \, u_1(t,0)+g(t),& \, \forall  \, t > 0, \\
u_j(t,\ell_j) = \partial_x u_j (t,\ell_j) = 0,& \, \forall \, t > 0, \, j= 1,\ldots,N, \\
u_j(0,x)=u_j^0(x), \, \forall \, x \in (0,\ell_j),&\,   j=1,\ldots,N,
\end{array}
\right.
\end{equation}
and it satisfies,
\begin{equation}\label{ineq_source}
\norm{\underline u}{\mathbb B}^2 
\leq C(T,N,L,\alpha) (\norm{\underline u^0}{\mathbb{L}^2(\mathcal T)}^2+\norm{\underline f}{L^1(0,T,\mathbb{L}^2(\mathcal T))}^2+\norm{g}{L^2(0,T)}^2).
\end{equation}
\end{proposition}

\begin{proof}
Thanks to Proposition \ref{well_posed_LKdV}, we consider that $(\underline u^0,g)=(\underline 0,0)$. By using standard semi-group theory, we get that if $f\in L^1(0,T,\mathbb{L}^2(\mathcal T))$ then the solution of \eqref{KdV-source} verifies $\underline u\in C([0,T], \mathbb L^2(\mathcal T))$ and 
\begin{equation*}
\norm{\underline u}{C([0,T], \mathbb L^2(\mathcal T))}\leq C\norm{f}{L^1(0,T,\mathbb{L}^2(\mathcal T))}.
\end{equation*}

As before we multiply the PDE in \eqref{KdV-source} by $\bar u_j$ and we integrate by parts on $[0,T]\times (0,\ell_j)$.
We easily obtain that,
\begin{equation*}
\begin{split}
\sum_{j=1}^N\int_0^{\ell_j}\abs{u_j(T,x)}^2dx+(2\alpha-N)\int_0^T\abs{u_1(t,0)}^2dt+\sum_{j=1}^N\int_0^T\abs{\partial_x u_j(t,0)}^2dt\leq 
\\
2\sum_{j=1}^N\int_0^T \norm{f_j(t,.)}{L^2(0,\ell_j)} \norm{u_j(t,.)}{L^2(0,\ell_j)}dt\\
\leq 2 \norm{\underline u}{C([0,T], \mathbb L^2(\mathcal T))}\norm{f}{L^1(0,T,\mathbb{L}^2(\mathcal T))}\\
\leq C\norm{f}{L^1(0,T,\mathbb{L}^2(\mathcal T))}^2.
\end{split}\end{equation*}

Thus 
\begin{equation}\label{estim_node}
\norm{u_1(t,0)}{L^2(0,T)}^2+\norm{\partial_x \underline u(t,0)}{L^2(0,T)}^2\\
\leq C\norm{f}{L^1(0,T,\mathbb{L}^2(\mathcal T))}^2.
\end{equation}

Next, we multiply the PDE in \eqref{KdV-source} by $x\bar u_j$ and integrate by parts. We obtain,
\begin{equation*}
\begin{split}
&\sum_{j=1}^N\int_0^{\ell_j}x\abs{u_j(T,x)}^2 dx-\sum_{j=1}^N\int_0^T\int_0^{\ell_j}\abs{u_j}^2dxdt+3\sum_{j=1}^N\int_0^T\int_0^{\ell_j}\abs{\partial_xu_j}^2dxdt\\
&=2\sum_{j=1}^N \int_0^T\int_0^{\ell_j}x\bar u_jf_jdxdt+\sum_{j=1}^N\int_0^T\bigg(-2u_j\partial_x \bar{u}_j \bigg)(t,0)dt.
\end{split}
\end{equation*}
Thanks to \eqref{estim_node}, we obtain that,
\begin{equation*}\begin{split}
3\sum_{j=1}^N\int_0^T\int_0^{\ell_j}\abs{\partial_xu_j}^2dxdt&\leq T\norm{\underline u}{C([0,T], \mathbb L^2(\mathcal T))}^2+2L\norm{\underline u}{C([0,T], \mathbb L^2(\mathcal T))}\norm{f}{L^1(0,T,\mathbb{L}^2(\mathcal T))}\\
&\,+\norm{u_1(t,0)}{L^2(0,T)}^2+\norm{\partial_x \underline u(t,0)}{L^2(0,T)}^2\\
&\leq C\norm{f}{L^1(0,T,\mathbb{L}^2(\mathcal T))}^2.
\end{split}\end{equation*}

Which ends the proof.
\end{proof}

\subsection{Well-posedness of $(KdV)$ and regularity results}

In order to prove the well-posedness of the nonlinear KdV equation, we need some regularity on the nonlinearity appearing in the equation and at the central node.

\medskip

\noindent We first recall the following Proposition whose proof can be found in \cite[Proposition 4.1]{Rosier} or \cite[Proposition 4]{Cerpa}.

\begin{proposition}\label{uu_x}
Let $T,L>0$, let $y\in L^2(0,T;H^1(0,L))$. Then $yy_x\in L^1(0,T;L^2(0,L))$ and the map $y\in L^2(0,T;H^1(0,L))\mapsto yy_x\in L^1(0,T;L^2(0,L))$ is continuous. Moreover, we have 
\begin{equation}
\norm{yy_x}{L^1(0,T;L^2(0,L))}\leq C\norm{y}{L^2(0,T;H^1(0,L))}^2.
\end{equation}

\end{proposition}

We also need the following proposition,
\begin{proposition}\label{u(.,0)}
Let $\underline u\in \mathbb B$, then $\abs{u_1(.,0)}^2\in L^2(0,T)$ and the map $\underline u\in \mathbb B\mapsto \abs{u_1(.,0)}^2\in L^2(0,T)$ is continuous.  Moreover, we have the estimate,
\begin{equation}\label{u_1}
\norm{u_1^2(.,0)}{L^2(0,T)}\leq \frac{1}{\sqrt{2}}\norm{\underline u}{\mathbb{B}}^2.
\end{equation}\end{proposition}

\begin{proof}
Let $\underline u,\underline v\in \mathbb{B}$. As $u_1(t,\ell_1)=v_1(t,\ell_1)=0$ we have 
\begin{equation*}
\begin{split}
\abs{u_1^2(t,0)-v_1^2(t,0)}=\left|\frac{1}{2}\int_0^{\ell_1} u_1(t,x)\partial_x u_1(t,x)- v_1(t,x)\partial_x v_1(t,x)dx \right|\\
\leq\frac{1}{2}\int_0^{\ell_1} \abs{(u_1-v_1)\partial_x u_1+v_1(\partial_x u_1-\partial_x v_1)}dx\\
\leq \frac{1}{2}\left(\norm{u_1(t,.)-v_1(t,.)}{L^2(0,\ell_1)}\norm{ u_1(t,.)}{H^1(0,\ell_1)}+\norm{v_1(t,.)}{L^2(0,\ell_1)}\norm{ u_1(t,.)-v_1(t,.)}{H^1(0,\ell_1)}\right).
\end{split}
\end{equation*}
Thus,
$$
\int_0^T\abs{u_1^2(t,0)-v_1^2(t,0)}^2 dt 
$$
$$
\leq \int_0^T\frac{1}{2}\left(\norm{u_1(t,.)-v_1(t,.)}{L^2(0,\ell_1)}^2\norm{ u_1(t,.)}{H^1(0,\ell_1)}^2+\norm{v_1(t,.)}{L^2(0,\ell_1)}^2\norm{ u_1(t,.)-v_1(t,.)}{H^1(0,\ell_1)}^2\right)dt
$$
$$
\leq \frac{1}{2}\norm{u_1-v_1}{C([0,T];L^2(0,\ell_1))}^2\norm{u_1}{L^2(0,T;H^1(0,\ell_1))}^2+\frac{1}{2}\norm{v_1}{C([0,T];L^2(0,\ell_1))}^2\norm{u_1-v_1}{L^2(0,T;H^1(0,\ell_1))}^2
$$
$$
\leq \frac{1}{2}(\norm{\underline u}{\mathbb{B}}^2+\norm{\underline v}{\mathbb{B}}^2)\norm{\underline u-\underline v}{\mathbb{B}}^2.
$$
We  get the desired result and estimate \eqref{u_1}.\end{proof}

With those both previous Propositions, we can prove the well-posedness of the non-linear KdV system.

\begin{theorem}\label{well-posed KdV}
Let $(\ell_i)_{i=1..N}\in (0,+\infty)^N$  and $T>0$. Then there exist $\epsilon>0$ and $C>0$ such that for all $\underline u^0\in \mathbb{L}^2(\mathcal{T})$ with $\norm{\underline u^0}{\mathbb{L}^2(\mathcal{T})}<\epsilon$, then there exists a unique solution of (KdV) that satisfies $\norm{\underline u}{\mathbb B}\leq C\norm{\underline u^0}{\mathbb{L}^2(\mathcal{T})}$.
\end{theorem}
\begin{proof} Let us fix $\underline u^0\in \mathbb{L}^2(\mathcal T)$ such that $\norm{\underline u^0}{\mathbb{L}^2(\mathcal{T})}<\epsilon$ where $\epsilon>0$ will be chosen later. We prove this theorem by using the Banach fixed point Theorem on the following map,
$F:\underline u\in \mathbb B\mapsto \underline v\in \mathbb B$, where $\underline v$ is the solution of,
\begin{equation}\label{KdV-g}
\left \{
\begin{array}{ll}
(\partial_t v_{j} + \partial_x v_{j} + \partial^3_x v_j)(t,x)=-u_j\partial_x u_j,\,& \forall \, x \in(0,\ell_j),\, t>0,\, j = 1,\ldots,N, \\
v_{j}(t,0)=v_{k}(t,0),\, &\forall \, j,k = 1,\ldots,N, \, t > 0, \\
\ds \sum_{j=1}^N \partial^2_{x} v_j (t,0) = - \alpha \, v_1(t,0)-\frac{N}{3}(u_1(t,0))^2, \,& \forall  \, t > 0, \\
v_j(t,\ell_j) = \partial_x v_j (t,\ell_j) = 0, \, &\forall \, t > 0, \, j= 1,\ldots,N, \\
v_j(0,x)=u_j^0(x), \, \forall \, x \in (0,\ell_j),\, &  j=1,\ldots,N.
\end{array}
\right.
\end{equation}
Clearly, $\underline u\in \mathbb B$ is a solution of (KdV) is equivalent to $\underline u$ is a fixed point of $F$. By using the previous regularity results,  namely Propositions \ref{well_posed_source_term}, \ref{uu_x} and \ref{u(.,0)}, we get that for all $\underline u\in \mathbb B$,
\begin{equation*}
\norm{F\underline u}{\mathbb B}\leq C \left(\norm{\underline u^0}{\mathbb{L}^2(\mathcal{T})}+\norm{\underline u}{\mathbb B}^2 \right),
\end{equation*} 
and for all $\underline u^1, \underline u^2\in \mathbb B$,

\begin{equation*}
\norm{F\underline u^1-F\underline u^2}{\mathbb B}\leq C \left(\norm{\underline u^1}{\mathbb B}+\norm{\underline u^2}{\mathbb B} \right)\norm{\underline u^1-\underline u^2}{\mathbb B}.
\end{equation*} 
Let us choose $R>0$ to be defined later and $\underline u$, $\underline u^1$ and $\underline u^2\in B_{\mathbb B}(0,R)$, then we have
\begin{equation*}\begin{split}
&\norm{F\underline u}{\mathbb B}\leq C(\epsilon+R^2)\\
&\norm{F\underline u^1-F\underline u^2}{\mathbb B}\leq C(2R)\norm{\underline u^1-\underline u^2}{\mathbb B}.
\end{split}
\end{equation*}
Thus by taking $R>0$ such that $R<\frac{1}{2C}$ and $\epsilon>0$ such that $C(\epsilon+R^2)<R$ we get the well-posedness result with the Banach fixed point Theorem.
\end{proof}

\section{Exponential stability}\label{Stability}
\subsection{Exponential stability of $(LKdV)$} 

In this section we will study two cases. First when the number of lengths which are in the  space of critical lengths, namely $\mathcal N:=\{2\pi \sqrt{\frac{k^2+l^2+kl}{3}}, \, k,l\in \mathbb N^*\}$, is strictly less than two. And in the second case when this number is larger than two.

\subsubsection{Observability inequality and stability in the non critical case}

\begin{theorem}\label{inequality_observability}
Let $(\ell_i)_{i=1..N}\in (0,+\infty)^N$ such that $\# \{\ell_i\in \mathcal N\}\leq 1$. Then for all $T >0$, there exists $C>0$ such that for all $\underline u^0\in \mathbb{L}^2(\mathcal T)$ we have,
\begin{equation}\label{observability}
\norm{\underline u^0}{L^2(\mathcal T)}^2\leq C \, \left(\norm{\partial_x\underline u(.,0)}{L^2(0,T)}^2+ \left(\alpha-\frac{N}{2} \right) \norm{\underline u(.,0)}{L^2(0,T)}^2 \right),
\end{equation}
where $\underline u\in \mathbb B$ is the solution of $(LKdV)$.
\end{theorem}

\begin{proof}

We follow the proof of Lemma 3.5 in \cite{Rosier} or Proposition 8 in \cite{Cerpa}. Let us suppose that the result is false. Then we could find a sequence $(\underline u^{0,n})_{n\in \mathbb N}\in \mathbb{L}^2(\mathcal T)$ such that $\norm{\underline u^{0,n}}{\mathbb{L}^2(\mathcal T)}=1$ and such that 
$$\norm{\partial_x\underline u^n(.,0)}{L^2(0,T)}^2+\norm{\underline u^n(.,0)}{L^2(0,T)}^2 \rightarrow 0$$
where $\underline u^n:=S(.)\underline u^{0,n}$.

By using estimates \eqref{L2H1etCL2} we have
\begin{equation*}
\norm{\underline u^n}{L^2(0,T,\mathbb H^1_e(\mathcal T))}\leq \norm{\underline u^n}{\mathbb B}\leq C(T,L,N,\alpha).
\end{equation*}
Thus $(\underline u^n)$ is bounded in $L^2(0,T,\mathbb H^1_e(\mathcal T))$ and then $(\underline u_t^n)$ is bounded in $L^2(0,T,\mathbb H^{-2}_e(\mathcal T))$. Thanks to the Aubin-Lions Lemma, we can deduce that $\underline u^n$ is relatively compact in $L^2(0,T,\mathbb L^2(\mathcal T))$ and we can assume that $\underline u^n$ converges in $L^2(0,T,\mathbb L^2(\mathcal T))$.

With inequality \eqref{u^0_bis}, we have 
\begin{equation*}
\norm{\underline u^{0,n}}{\mathbb{L}^2(\mathcal{T})}^2\leq\frac{1}{T}\norm{\underline u^n}{L^2(0,T,\mathbb{L}^2(\mathcal T))}^2+ 3\left(\alpha-\frac{N}{2} \right)\norm{ u_1^n(.,0)}{L^2(0,T)}^2+\norm{\partial_x\underline u^n(.,0)}{L^2(0,T)}^2.
\end{equation*}
As the two  last terms tends to $0$ as $n$ tends to infinity, $(\underline u^{0,n})$ is a Cauchy sequence in $\mathbb{L}^2(\mathcal{T})$ and then converges to a function $\underline u^0$ satisfying $\norm{\underline u^{0}}{\mathbb{L}^2(\mathcal T)}=1$. Then, we have $\underline u=S(.)\underline u^0$, $u_1(t,0)=0$ and $\partial_x \underline u(t,0)=\underline 0$.

With the same type of proof as in \cite{Rosier}, we have to prove the following Lemma:

\begin{lemma}
Let  $(\ell_i)_{i=1..N}\in (0,+\infty)^N$. Let us consider the following assertion:

\begin{equation}\label{assertion} \exists \, (\lambda_i)_{i=1\ldots N}\in \mathbb C^N, \, \exists \underline y\in \left[\prod_{i=1}^N H^3(0,\ell_i) \right] \setminus\{\underline 0\} \, s.t. \, \left\{\begin{array}{l}
\lambda_i y_i+y_i'+y_i'''=0,\, \forall i=1\ldots N, \\
y_i(\ell_i)=0,\, y_i'(\ell_i)=0,\, \forall i=1\ldots N, \\
y_i(0)=0,\, y_i'(0)=0,\, \forall i=1\ldots N, \\
\ds \sum_{i=1}^N y_i''(0)=0.
\end{array}\right. \end{equation}
Then $\eqref{assertion}\Leftrightarrow \# \{\ell_i\in \mathcal N\}\geq 2$.
\end{lemma}

\begin{proof} 

Let us first recall Lemma 3.5 in \cite{Rosier}:
\begin{lemma}\cite[Lemma 3.5]{Rosier}\label{lemma_rosier}
Let $L\in (0,+\infty)$. Consider the following assertion,
\begin{equation}\label{assertion_rosier}
 \exists \, \lambda\in \mathbb C, \, \exists  y\in  H^3(0,L) \setminus\{ 0\} \, s.t. \, \left\{\begin{array}{l}
\lambda y+y'+y'''=0, \\
y(0)=0,\, y'(0)=0,\,y(L)=0,\, y'(L)=0.
\end{array}\right.
\end{equation}
Then $\eqref{assertion_rosier}\Leftrightarrow L\in \mathcal N$.
\end{lemma}

\begin{enumerate}
\item If $\forall i,\, \ell_i\notin \mathcal N$ then Lemma \ref{lemma_rosier} gives us $\underline y=0$.
\item If $\# \{\ell_i\in \mathcal N\}=1$, then we can suppose  that $\ell_1\in \mathcal N$ and for all $i=2\ldots N,$ $\ell_i\notin \mathcal N$. Then Lemma  \ref{lemma_rosier} gives us that for all $i=2\ldots N,$ $\ell_i\notin \mathcal N$, $y_i=0$ and then $y_1$ has to satisfy, 
$$\left\{\begin{array}{l}
\lambda_1 y_1+y_1'+y_1'''=0, \\
y_1(\ell_1)=0,\, y_1'(\ell_1)=0,\\
y_1(0)=0,\, y_1'(0)=0, \\
 y_1''(0)=0.
\end{array}\right. $$
Due to the three null conditions at the spatial origin $0$, the unique solution of this system is $y_1=0$. 

Thus $\underline y=0$.

\item If $\# \{\ell_i\in \mathcal N\}\geq 2$, then we can suppose  that $\ell_1,\, \ell_2\in \mathcal N$. Then we can take $y_i=0$ for $i=3\ldots N$.
Lemma \ref{lemma_rosier}  gives us two  non null functions $z_1$ and $z_2$ that satisfy 
$$\left\{\begin{array}{l}
\lambda_1 z_1+z_1'+z_1'''=0, \\
z_1(\ell_1)=0,\, z_1'(\ell_1)=0,\\
z_1(0)=0,\, z_1'(0)=0,
\end{array}\right. 
\left\{\begin{array}{l}
\lambda_2 z_2+z_2'+z_2'''=0, \\
z_2(\ell_2)=0,\, z_2'(\ell_2)=0,\\
z_2(0)=0,\, z_2'(0)=0.
\end{array}\right. $$
We then define $$\underline y=(z_2''(0)z_1,-z_1''(0)z_2,0,\ldots,0).$$  As $z_1$ and $z_2$ are non null satisfy an ODE of order 3 and $z_1(0)=z_1'(0)=0$ and $z_2(0)=z_2'(0)=0$ then $z_1''(0)\neq 0$ and $z_2''(0)\neq 0$. Then $\underline y$ is non null and  satisfies the system given in \eqref{assertion}.

\end{enumerate}
\end{proof}
From this Lemma, we easily deduce the observability inequality \eqref{observability} and this ends the proof of Theorem \ref{inequality_observability}.
\end{proof}

We can now prove the result of stability.

\begin{theorem}\label{stability_non_critical}
Let $(\ell_i)_{i=1..N}\in (0,+\infty)^N$ such that $\# \{\ell_i\in \mathcal N\}\leq 1$, then there exists $C>0$ and $\mu>0$ such that for all $\underline u^0\in \mathbb{L}^2(\mathcal T)$ the solution of $(LKdV)$ satisfies,
\begin{equation*}
\norm{\underline u(t,.)}{\mathbb{L}^2(\mathcal T)}\leq C \, \norm{\underline u^0}{\mathbb{L}^2(\mathcal T)}e^{-\mu t}.
\end{equation*}
\end{theorem}

\begin{proof} We follow the proof given in \cite{Perla}. With \eqref{dissipae1} we have by integration and using the previous observability inequality \eqref{observability},
\begin{equation*}
\begin{split}
\norm{\underline u(T,.)}{\mathbb{L}^2(\mathcal T)}^2&=\norm{\underline u^0}{\mathbb{L}^2(\mathcal T)}^2- \left( \left(\alpha-\frac{N}{2} \right)\norm{ u_1(.,0)}{L^2(0,T)}^2+\norm{\partial_x\underline u(.,0)}{L^2(0,T)}^2 \right)\\
&\leq \frac{C-1}{C}\norm{\underline u^0}{\mathbb{L}^2(\mathcal T)}^2
\end{split}
\end{equation*}
Thus we get easily the stability result.

\end{proof}

\subsubsection{Stability in the critical case}

We suppose in this section that $\# \{\ell_i\in \mathcal N\}\geq 2$ then  adding a damping mechanism on the critical branches except at most one gives the stability of the system.

Let us define $I_c=\{i\in\{1,\ldots,N\}, \ell_i\in \mathcal N\}$, the set of critical indexes,  and $I_c^*$ equals to $I_c$ minus one index.
We study the following problem,

\begin{equation*}
\leqno(LKdV_{damped}) 
\left \{
\begin{array}{ll}
(\partial_t u_{j} + \partial_x u_{j} + \partial^3_x u_j+a_j(x)u_j)(t,x)=0,&\forall \, x \in(0,\ell_j),\, t>0,\, j = 1,...,N, \\
u_{j}(t,0)=u_{k}(t,0),& \forall \, j,k = 1,...,N, \, t > 0, \\
 \ds \sum_{j=1}^N \partial^2_{x} u_j (t,0) = - \alpha \, u_1(t,0), &\forall \, t > 0, \\
u_j(t,\ell_j) = \partial_x u_j (t,\ell_j) = 0, &\forall \, t > 0, \, j= 1,...,N, \\
u_j(0,x)=u_j^0(x), \, \forall \, x \in (0,\ell_j),&   j=1,...,N,
\end{array}
\right.
\end{equation*}

where $\alpha > \frac{N}{2}$ and the damping $(a_j)_{j=1, N} \in \ds \prod_{j=1}^N L^\infty (0,\ell_j)$ is defined by 
\begin{equation}\label{a1}\left\{\begin{array}{l}
a_j = 0 \mbox{ for } j\in \{1,\ldots,N\}\setminus I_c^*,\\
a_j \geq c_j \mbox{ in an open nonempty set } \omega_j \mbox{ of } (0,\ell_j), \mbox{ for all }  j\in I_c^*, \\
\hbox{and} \; c_j > 0 \; \hbox{is a constant}.
\end{array}
\right.\end{equation}

We can prove the well-posedness of this system as in \cite{Perla}, by considering it as a perturbation of (LKdV). With same types of arguments we get the stability result.

\begin{theorem}\label{stability_critical}
Assume that the damping $\underline a$ is defined as in \eqref{a1}, then there exist $C>0$ and $\mu>0$ such that for all $\underline u^0\in \mathbb{L}^2(\mathcal T)$, the solution of $(LKdV_{damped})$ satisfies,
\begin{equation*}
\norm{\underline u(t,.)}{\mathbb{L}^2(\mathcal T)}\leq C\norm{\underline u^0}{\mathbb{L}^2(\mathcal T)}e^{-\mu t}.
\end{equation*}

\end{theorem}

\begin{proof} We first multiply the PDE of $(LKdV_{damped})$ by $x\overline{u_j}$ and we easily get the following estimate,

\begin{equation*}
\norm{\partial_x \underline u}{L^2([0,T],\mathbb{L}^2(\mathcal T))}^2\leq \frac{1}{3} \, \left(L+T+N+\frac{1}{2\alpha-N} \right) \, \norm{\underline u^0}{\mathbb{L}^2(\mathcal T)}^2.
\end{equation*}

Then we multiply the PDE of $(LKdV_{damped})$ by $(T-t) \overline{u_j}$ to get,
\begin{equation}\label{u^0}
\begin{split}
\norm{\underline u^0}{\mathbb{L}^2(\mathcal{T})}^2&\leq\frac{1}{T}\norm{\underline u}{L^2(0,T,\mathbb{L}^2(\mathcal T))}^2+(2\alpha-N)\norm{\underline u(.,0)}{L^2(0,T)}^2\\
&+\norm{\partial_x\underline u(.,0)}{L^2(0,T)}^2+2\sum_{j\in I_c^*}\int_0^T\int_0^{l_j}a_j(x)\abs{u_j}^2dxdt.
\end{split}\end{equation}

We argue by contradiction to prove the following inequality,
\begin{equation*}
\norm{\underline u^0}{\mathbb{L}^2(\mathcal{T})}^2\leq C\left((2\alpha-N)\norm{\underline u(.,0)}{L^2(0,T)}^2\\
+\norm{\partial_x\underline u(.,0)}{L^2(0,T)}^2+2\sum_{j\in I_c^*}\int_0^T\int_0^{l_j}a_j(x)\abs{u_j}^2dxdt\right).
\end{equation*}

By following the same arguments as for the proof of Theorem \ref{inequality_observability} we can construct a sequence $(\underline u^{0,n})\in \mathbb{L}^2(\mathcal T)$ such that the corresponding solution of $(LKdV_{damped})$ satisfies
$$\left\{\begin{array}{l}
 \norm{\underline u^n(.,0)}{L^2(0,T)}\rightarrow 0,\\
 \norm{\partial_x\underline u^n(.,0)}{L^2(0,T)}\rightarrow 0,\\ 
\ds \sum_{j\in I_c^*}\int_0^T\int_0^{l_j}a_j(x)\abs{u^n_j}^2dxdt\rightarrow 0.
 \end{array}\right.$$
  By passing to the limit we obtain a non trivial solution $\underline u\in \mathbb B$ of $(LKdV_{damped})$ such that 
  $$\left\{\begin{array}{l}
  \underline u(.,0)=\underline 0,\\
  \partial_x\underline u(.,0)=\underline 0,\\
\int_0^T\int_0^{l_j}a_j(x)\abs{u_j}^2dxdt=0,\,  \forall j\in I_c^*.
   \end{array}\right.$$

\begin{enumerate}
\item For all $j\in \{1,\ldots,N\}\setminus I_c$, $u_j$ is solution of $(LKdV)$ and such that $u_j(.,0)=\partial_x u_j(.,0)=0$. Then thanks to Lemma \ref{lemma_rosier}, $u_j=0$. 
\item For all $j\in I_c^*$, $\int_0^T\int_0^{l_j}a_j(x)\abs{u_j}^2dxdt=0$, thus $a_ju_j=0$ and $u_j=0$ in $(0,T)\times \omega_j$. Then $\partial_t u_j+\partial_x u_j+\partial^3_x u_j=0$ and thanks to Holmgren's Theorem, $u_j=0$.
\item For $j\in I_c\setminus I_c^*$, $u_j$ satisfies,
\begin{equation*}
\left\{\begin{array}{l}
\partial_t u_j+\partial_x u_j+\partial^3_x u_j=0,\\
u_j(t,0)=0,\, \partial_x u_j(t,0)=0, \partial^2_x u_j(t,0)=0,\\
u_j(t,\ell_j)=\partial_x u_j(t,\ell_j)=0.
\end{array}\right.
\end{equation*} 
Due to the three null conditions at the central node, we obtain that $u_j=0$.
\end{enumerate}
Thus $\underline u=\underline 0$ and we get a contradiction which ends the proof of Theorem \ref{stability_critical}. 

\end{proof}

\subsection{Stabilization of the $(KdV)$ system on a star-shaped network in the critical or non critical case} \label{nonlinear}

\subsubsection{Stability for small amplitude solutions}
In this section we study the stabilization of the non linear $(KdV)$ system for the critical and the non critical case.

We define as before  $I_c=\{i\in\{1,\ldots,N\}, \ell_i\in \mathcal N\}$, the set of critical indexes,  and $I_c^*$ equals to $I_c$ minus one index. Eventually, $I_c^*=\emptyset$.
We study the following problem,

\begin{equation*}
\leqno(KdV_{damped}) 
\left \{
\begin{array}{ll}
(\partial_t u_{j} + \partial_x u_{j} + \partial^3_x u_j+a_ju_j+u_j\partial_x u_j)(t,x)=0,& \forall \, x \in(0,\ell_j),\, t>0,\, j = 1,...,N, \\
u_{j}(t,0)=u_{k}(t,0),& \forall \, j,k = 1,...,N, \, t > 0, \\
 \ds \sum_{j=1}^N \partial^2_{x} u_j (t,0) = - \alpha \, u_1(t,0)-\frac{N}{3}(u_1(t,0))^2,&  \forall \, t > 0, \\
u_j(t,\ell_j) = \partial_x u_j (t,\ell_j) = 0,&  \forall \, t > 0, \, j= 1,...,N, \\
u_j(0,x)=u_j^0(x), \, \forall \, x \in (0,\ell_j),&   j=1,...,N,
\end{array}
\right.
\end{equation*}

where $\alpha > \frac{N}{2}$ and the damping $(a_j)_{j=1, N} \in \ds \prod_{j=1}^N L^\infty (0,\ell_j)$ is defined by 
\begin{equation}\label{a}\left\{\begin{array}{l}
a_j = 0 \mbox{ for } j\in \{1,\ldots,N\}\setminus I_c^*,\\
a_j \geq c_j \mbox{ in an open nonempty set } \omega_j \mbox{ of } (0,\ell_j), \mbox{ for }  j\in I_c^*, \\
\hbox{and} \; c_i > 0 \; \hbox{is a constant}.
\end{array}
\right.\end{equation}

Let $\underline u^0\in \mathbb L^2(\mathcal T)$  such that $\norm{u^0}{\mathbb L^2(\mathcal T)}$ is sufficiently small in order to have with Theorem \ref{well-posed KdV}  the existence and unicity  of $\underline u\in \mathbb B$  solution of $(KdV_{damped})$ which is a perturbation of $(KdV)$. Then we can decompose $\underline u$ into $\underline u^1+\underline u^2$ respective solutions of 
\begin{equation*}
\left \{
\begin{array}{ll}
(\partial_t u^1_{j} + \partial_x u^1_{j} + \partial^3_x u^1_j+a_ju^1_j)(t,x)=0,& \forall \, x \in(0,\ell_j),\, t>0,\, j = 1,\ldots,N, \\
u^1_{j}(t,0)=u^1_{k}(t,0),& \forall \, j,k = 1,\ldots,N, \, t > 0, \\
\ds \sum_{j=1}^N \partial^2_{x} u^1_j (t,0) = - \alpha \, u^1_1(t,0), & \forall  \, t > 0, \\
u^1_j(t,\ell_j) = \partial_x u^1_j (t,\ell_j) = 0, & \forall \, t > 0, \, j= 1,\ldots,N, \\
u^1_j(0,x)=u_j^0(x), \,& \forall \, x \in (0,\ell_j),  j=1,\ldots,N,
\end{array}\right.
\end{equation*}
\begin{equation*}
\left \{
\begin{array}{ll}
(\partial_t u^2_{j} + \partial_x u^2_{j} + \partial^3_x u^2_j+a_ju^2_j)(t,x)=-u_j\partial_x u_j,& \forall \, x \in(0,\ell_j),\, t>0,\, j = 1,\ldots,N, \\
u^2_{j}(t,0)=u^2_{k}(t,0),& \forall \, j,k = 1,\ldots,N, \, t > 0, \\
\ds \sum_{j=1}^N \partial^2_{x} u^2_j (t,0) = - \alpha \, u^2_1(t,0)-\frac{N}{3}(u_1(t,0))^2,&  \forall  \, t > 0, \\
u^2_j(t,\ell_j) = \partial_x u^2_j (t,\ell_j) = 0,& \forall \, t > 0, \, j= 1,\ldots,N, \\
u^2_j(0,x)=0, \,& \forall \, x \in (0,\ell_j),  j=1,\ldots,N.
\end{array}\right.
\end{equation*}

Then thanks to Theorems \ref{stability_non_critical} and \ref{stability_critical} we have the existence of $\gamma<1$ such that for all $t\in [0,T]$,
\begin{equation*}
\norm{\underline u^1(t,.)}{\mathbb L^2(\mathcal T)}\leq \gamma \norm{\underline u^0}{\mathbb L^2(\mathcal T)}.
\end{equation*} 

Thanks to Propositions   \ref{well_posed_source_term}, \ref{uu_x} and \ref{u(.,0)}  we can deduce that 
\begin{equation*}\begin{split}
\norm{\underline u^2(t,.)}{\mathbb L^2(\mathcal T)}&\leq C(\norm{\underline u\partial_x\underline u}{L^1(0,T,\mathbb L^2(\mathcal T))}+\norm{(u_1(t,0))^2}{L^2(0,T)}\\
&\leq C\norm{\underline u}{\mathbb B}^2.
\end{split}\end{equation*}

We need some estimates on this last right term.

We first multiply the equation of $(KdV_{damped})$ by $\bar u_j$ and integrate in space and time over $(0,s)$ to obtain  
\begin{equation*}
\norm{\underline u (s,.)}{\mathbb L^2(\mathcal T)}^2  + \int_0^s \sum_{j=1}^N \abs{\partial_x u_j(t,0)}^2 \, dt + (2\alpha-N) \int_0^s\abs{u(t,0)}^2dt+ 2\sum_{j\in I_c^*}\int_0^T\int_0^{l_j}a_j(x)\abs{u_j}^2dxdt= \norm{\underline u^0}{\mathbb L^2(\mathcal T)}^2.
\end{equation*}
Secondly, we multiply $(KdV_{damped})$ by $x\bar u$ and integrate in space and time and obtain with the previous result,
\begin{equation}\label{u_H^1}
\norm{\partial_x\underline u}{L^2(0,T,\mathbb L^2(\mathcal T))}^2\leq C(T,L,N,\alpha)\norm{\underline u^0}{\mathbb L^2(\mathcal T)}^2+\frac{2}{9}\int_0^T\int_\mathcal{T}(\underline u)^3dxdt.
\end{equation} 

As for all $i=1,\ldots,N$, $u_i\in L^2(0,T,H^1(0,\ell_i))$ and $H^1(0,\ell_i)$ embeds into $C([0,\ell_i])$, we have as in \cite{Cerpa} or \cite{Perla}, 
\begin{equation*}
\begin{split}
\ds \sum_{i=1}^N \int_0^T\int_0^{\ell_i}\abs{u_i}^3dxdt\leq CT^{1/2}\norm{\underline u_0}{\mathbb L^2(\mathcal T)}^2\norm{\underline u}{L^2(0,T;\mathbb H^1(\mathcal T))}.
\end{split}
\end{equation*}
We obtain with \eqref{u_H^1},
\begin{equation}\label{3.49}
\norm{\underline u}{L^2(0,T:\mathbb H^1(\mathcal T))}^2\leq C(T,L,N,\alpha) \left(\norm{\underline u_0}{\mathbb L^2(\mathcal T)}^2+\norm{\underline u_0}{\mathbb L^2(\mathcal T)}^4 \right).
\end{equation}

This gives with the previous inequalities, the estimate, 
\begin{equation*}
\norm{\underline u(s,.)}{\mathbb L^2(\mathcal T)}\leq \norm{\underline u_0}{L^2(0,T)}(\gamma+C\norm{\underline u_0}{\mathbb L^2(\mathcal T)}+C\norm{\underline u_0}{\mathbb L^2(\mathcal T)}^3).
\end{equation*}
Thus by taking $\epsilon >0$ small enough such that $\gamma+C\epsilon+C^3\epsilon<1$ if $\underline u_0$ satisfies $\norm{\underline u_0}{\mathbb L^2(\mathcal T))}<\epsilon$ we have 
\begin{equation*}
\norm{\underline u(s,.)}{\mathbb L^2(\mathcal T)}\leq (\gamma+C\epsilon+C^3\epsilon)\norm{\underline u_0}{\mathbb L^2(\mathcal T)}.
\end{equation*}
and we get the stability result.

\subsubsection{Semi-global stability result}
In this section we prove a semi-global result, provided that the damping is applied on all branches.

Let $\underline a\in \mathbb L^\infty(\mathcal T)$ with,
\begin{equation}\label{damping}
\left\{\begin{array}{l}
a_i(x)\geq a_0>0,\, \forall x\in \omega_i,\, \forall i=1,\ldots,N,\\
\mbox{ with } \omega_i \mbox{ a nonempty open subset of } (0,\ell_i). 
\end{array}\right.
\end{equation}
Then our main result of this section is:
\begin{theorem}
Let $(\ell_i)_{i=1,\ldots,N}\in (0,+\infty)^N$, let $\underline a\in \mathbb L^\infty(\mathcal T)$ satisfying \eqref{damping}, and let $R>0$. Then for all $\underline u^0\in \mathbb{L}^2(\mathcal T)$ with $\norm{\underline u^0}{\mathbb{L}^2(\mathcal T)}\leq R$ then there exist $C=C(R)>0$ and $\mu=\mu(R)>0$ such that  the solution $\underline u$ of $(KdV_{damped})$ satisfies,
$$\norm{\underline u(t,.)}{\mathbb{L}^2(\mathcal T)}\leq Ce^{-\mu t}\norm{\underline u^0}{\mathbb{L}^2(\mathcal T)},\, \forall t\geq 0.$$
\end{theorem}
\begin{proof}
To prove this result we follow the article of Pazoto \cite{Pazoto}.  Our result is based on this Unique Continuation Property of Saut and Sheurer \cite{Saut}.
\begin{theorem}\label{saut}(\cite[Theorem 4.2]{Saut})
Let $L>0$ and let $y\in L^2(0,T,H^3(0,L))$ be a solution of
\begin{equation*}
y_t+y_x+y_{xxx}+yy_x=0,
\end{equation*}
such that $y(t,x)=0$, $\forall t\in(t_1,t_2)$ and $x\in \omega$ where $\omega$ is a nonempty open subset of $(0,L)$. Then $y(t,x)=0$, $\forall t\in(t_1,t_2)$ and $x\in(0,L)$.
\end{theorem}

By multiplying $(KdV_{damped})$ by $\bar u_j$ and integrating on time and space, we have,
\begin{equation}\label{3_nonlin}
\norm{\underline u (s,.)}{\mathbb L^2(\mathcal T)}^2  + \int_0^s \sum_{j=1}^N \abs{\partial_x u_j(t,0)}^2 \, dt \, +   
\end{equation}
$$
(2\alpha-N) \int_0^s\abs{u_1(t,0)}^2dt+ 2\sum_{j=1}^N\int_0^T\int_0^{\ell_j}a_j(x)\abs{u_j}^2dxdt = \norm{\underline u^0}{\mathbb L^2(\mathcal T)}^2.
$$

By integrating \eqref{3_nonlin} over $(0,T)$ we have,
\begin{multline*}
T\norm{\underline u^0}{\mathbb L^2(\mathcal T)}^2\leq\int_0^T\norm{\underline u (s,.)}{\mathbb L^2(\mathcal T)}^2 ds 
+ T\int_0^T \sum_{j=1}^N \abs{\partial_x u_j(t,0)}^2 \, dt \\+ (2\alpha-N) T\int_0^T\abs{u_1(t,0)}^2+ 2T\sum_{j=1}^N\int_0^T\int_0^{\ell_j}a_j(x)\abs{u_j}^2dxdt.
\end{multline*} 
Thus we just have to prove that there exists $C=C(T,R)$ such that 
\begin{equation}
\int_0^T\norm{\underline u (t,.)}{\mathbb L^2(\mathcal T)}^2 dt \leq 
\end{equation}
$$
C\left(\int_0^T \sum_{j=1}^N \abs{\partial_x u_j(t,0)}^2 \, dt + (2\alpha-N) \int_0^T\abs{u_1(t,0)}^2dt+ 2\sum_{j=1}^N\int_0^T\int_0^{\ell_j}a_j(x)\abs{u_j}^2dxdt\right).
$$
\end{proof}

We assume that this inequality is false. Then we can find a sequence $(\underline u^n)\in \mathbb B$ solution of $(KdV_{damped})$ with $\norm{\underline u^{0,n}}{\mathbb L^2(\mathcal T)}\leq R$ and such that 
\begin{equation*}
\lim_{n\rightarrow \infty}\frac{\norm{\underline u^n}{L^2(0,T,\mathbb L^2(\mathcal T))}^2}{\norm{\partial_x\underline u^n(.,0)}{L^2(0,T)}^2+(2\alpha-N)\norm{ u^n_1(.,0)}{L^2(0,T)}^2+2\sum_{j=1}^N\int_0^T\int_0^{\ell_j}a_j(x)\abs{u_j^n}^2dxdt}=\infty.
\end{equation*}

Let us define $\lambda^n:=\norm{\underline u^n}{L^2(0,T,\mathbb L^2(\mathcal T))}$ and $\underline v^n:=\frac{\underline u^n}{\lambda^n}$. Then $\underline v^n$ satisfies the following problem,
\begin{equation}\label{eq_v}
\left\{\begin{array}{ll}
\partial_t v^n_i+\partial_x v^n_i+\partial_{xxx} v^n_i+a_i v^n_i+\lambda^n v_i^n\partial_x v_i^n=0,& i=1,\ldots,N\\
v^n_i(t,\ell_i)=0,\, \partial_x v^n_i(t,\ell_i)=0, & i=1,\ldots,N\\
v^n_i(t,0)=v^n_j(t,0),& i,j=1,\ldots,N,\\
\sum_{i=1}^N \partial_{xx}v^n_i(t,0)=-\alpha v_1^n(t,0)-\frac{N}{3}(v_1^n(t,0))^2, & i=1,\ldots,N\\
\norm{v^n}{L^2(0,T,\mathbb L^2(\mathcal T))}=1.
\end{array}
\right.
\end{equation}
By multiplying the PDE in \eqref{eq_v} by $\bar v_i$ and integrating on $(0,T)\times(0,\ell_i)$ we get
\begin{multline*}
T\norm{\underline v^n(0,.)}{\mathbb L^2(\mathcal T)}^2\leq \int_0^T\norm{\underline v^n (s,.)}{\mathbb L^2(\mathcal T)}^2 ds 
+ T\int_0^T \sum_{j=1}^N \abs{\partial_x v^n_j(t,0)}^2 \, dt \\+ (2\alpha-N) T\int_0^T\abs{v^n(t,0)}^2+ 2T\sum_{j=1}^N\int_0^T\int_0^{\ell_j}a_j(x)\abs{v^n_j}^2dxdt.
\end{multline*} 
Thus $(v^n(0,.))$ is bounded in $\mathbb L^2(\mathcal T)$.

By using \eqref{3_nonlin}, we see that $\lambda^n:=\norm{\underline u^n}{L^2(0,T,\mathbb L(\mathcal T))}\leq \sqrt T\norm{\underline u^{0,n}}{\mathbb L^2(\mathcal T)}\leq \sqrt TR$. Thus $(\lambda^n)$ is bounded in $\mathbb R$.

Then we can get as for the previous inequality \eqref{3.49},

\begin{equation}\label{3.58}
\norm{\underline v^n}{L^2(0,T:\mathbb H^1(\mathcal T))}^2\leq C(T,L,N,\alpha,R) \left(\norm{\underline v^{0,n}}{\mathbb L^2(\mathcal T)}^2+\norm{\underline v^{0,n}}{\mathbb L^2(\mathcal T)}^4 \right).
\end{equation}
Thus $(\underline v^n)$ is bounded in $L^2(0,T:\mathbb H^1(\mathcal T))$, and we can prove that for all $i=1,\ldots,N$, $(v^n_i\partial_x v^n_i)$ is a sequence of $L^2(0,T,L^1(0,\ell_i))$ as 
\begin{equation*}
\norm{v^n_i\partial_x v^n_i}{L^2(0,T,L^1(0,\ell_i))}\leq \norm{\underline v^n}{C([0,T],\mathbb L^2(\mathcal T))}\norm{\underline v^n}{L^2(0,T,\mathbb H^1(\mathcal T))}.
\end{equation*}

Thus we can deduce that $(\partial_t \underline v^n)$ is bounded in $L^2(0,T,\mathbb H^1(\mathcal T))$ and then we can extract from $(\underline v^n)$ a subsequence that converges strongly in $L^2(0,T,\mathbb L^2(\mathcal T))$ to a limit $\underline v$ with $\norm{\underline v}{L^2(0,T,\mathbb L^2(\mathcal T))}=1$ and we have $v_i(t,x)=0$, $ \forall x\in \omega_i$,  $v_i(t,0)=0$ and $\partial_x v_i(t,0)=0$, $\forall t\in (0,T),\, \forall i=1,\ldots,N$. 

As $(\lambda^n)$ is bounded in $\mathbb R$ we can extract a sequence that converges in $\mathbb R$ to a limit $\lambda\geq 0$. Thus $\underline v$ satisfies the following system,
\begin{equation*}
\left\{\begin{array}{ll}
\partial_t v_i+\partial_x v_i+\partial_{xxx} v_i+\lambda v_i\partial_x v_i=0,&\forall i=1,\ldots, N,\\
v_i(t,\ell_i)=0,\, \partial_x v_i(t,\ell_i)=0,\\
v_i(t,0)=0,\, \partial_x v_i(t,0)=0,\\
\norm{\underline v(0,.)}{\mathbb L^2(\mathcal T)}\leq R,\\
\norm{\underline v}{L^2(0,T,\mathbb L^2(\mathcal T))}=1.
\end{array}
\right.
\end{equation*}

\begin{enumerate}
\item If $\lambda=0$ then thanks to Holmgren's Theorem, we deduce that $\underline v=0$ which is absurd.
\item If $\lambda>0$ then we will apply the results of Saut and Sheurer \cite{Saut} to get a contradiction. As $v_i$ satisfies the same  equation as in \cite{Pazoto} we can deduce that $v_i\in L^2(0,T,H^3(0,\ell_i))$ for all $i=1,\ldots, N$. Thus by applying Theorem \ref{saut} we get the contradiction and then the stability result.
\end{enumerate}

\section{Controllability results.} \label{Control}

We first consider the following exact boundary controllability problem for the linearized KdV equation:

 For any $T>0$, $\alpha>\frac{N}{2}$ and $(\ell_i)_{i=1,\ldots,N}\in (0,+\infty)^N$, for every $\underline u^0,\underline u^T\in \mathbb L^2(\mathcal T)$, does there exist (N+1) controls $\underline g\in \mathbb{L}^2(0,T)$ and $g\in L^2(0,T)$ such that the solution $\underline u\in \mathbb B$ of the following system, $(LKdV_{control})$, satisfies $\underline u(0,.)=\underline u^0$ and $\underline u(T,.)=\underline u^T$ ?

\begin{equation*}
\leqno(LKdV_{control}) 
\left \{
\begin{array}{ll}
(\partial_t u_{j} + \partial_x u_{j} + \partial^3_x u_j)(t,x)=0,&\, \forall \, x \in(0,\ell_j),\, t>0,\, j = 1,...,N, \\
u_{j}(t,0)=u_{k}(t,0),&\, \forall \, j,k = 1,...,N, \, t > 0, \\
\ds \sum_{j=1}^N \partial^2_{x} u_j (t,0) = - \alpha \, u_1(t,0) +g(t),& \, \forall \, t > 0, \\
u_j(t,\ell_j) = 0,& \, \forall \, t > 0, \, j= 1,...,N,\\
\partial_x u_j (t,\ell_j) = g_j(t),& \, \forall \, t > 0, \, j= 1,...,N, \\
u_j(0,x)=u_j^0(x),& \, \forall \, x \in (0,\ell_j),\,   j=1,...,N.
\end{array}
\right.
\end{equation*}

By applying the Hilbert Uniqueness Method, \cite{Lions}, it is well known that the exact boundary controllability is equivalent to the inequality of observability for the following backward adjoint problem.
\begin{equation*}
\left \{
\begin{array}{ll}
(\partial_t \varphi_{j} + \partial_x \varphi_{j} + \partial^3_x \varphi_j)(t,x)=0,&\, \forall \, x \in(0,\ell_j),\, t>0,\, j = 1,...,N, \\
\varphi_{j}(t,0)=\varphi_{k}(t,0),&\, \forall \, j,k = 1,...,N, \, t > 0, \\
\partial_x\varphi_j(t,0)=0& \, \forall \, t > 0, \, j= 1,...,N, \\
\ds \sum_{j=1}^N \partial^2_{x} \varphi_j (t,0) = (\alpha-N)\varphi_1(t,0) ,& \, \forall \, t > 0, \\
\varphi_j(t,\ell_j) = 0,& \, \forall \, t > 0, \, j= 1,...,N,\\
\varphi_j(T,x)=\varphi_j^T(x),& \, \forall \, x \in (0,\ell_j),\,   j=1,...,N.
\end{array}
\right.
\end{equation*}
By following the same steps as done for Theorem \ref{inequality_observability}, we can prove this observability inequality,
\begin{theorem}
Let $(\ell_i)_{i=1..N}\in (0,+\infty)^N$ such that $\# \{\ell_i\in \mathcal N\}\leq 1$. Then for all $T >0$, there exists $C>0$ such that for all $\underline \varphi^T\in \mathbb{L}^2(\mathcal T)$ we have,
\begin{equation}\label{observabilitybis}
\norm{\underline \varphi^T}{\mathbb{L}^2(\mathcal T)}^2\leq C \left(\sum_{j=1}^N\norm{\partial_x \varphi_j(.,\ell_j)}{L^2(0,T)}^2+\int_0^T \varphi_1^2(t,0)dt\right),
\end{equation}
where $\underline \varphi\in \mathbb B$ is the solution of the backward adjoint problem.
\end{theorem}

Thus  we get the following exact boundary controllability result, provided that the network is non critical.

\begin{theorem}
Let $T>0$ and $(\ell_i)_{i=1,\ldots,N}\in (0,+\infty)^N$ such that $\# \{\ell_i\in \mathcal N\}\leq 1$. Then  for all $\underline u^0,\underline u^T\in \mathbb L^2(\mathcal T)$, there exists $\underline g\in \mathbb{L}^2(0,T)$ and $g\in L^2(0,T)$ such that the solution $\underline u\in \mathbb B$ of $(LKdV_{control})$ satisfies $\underline u(0,.)=\underline u^0$ and $\underline u(T,.)=\underline u^T$.
\end{theorem}

By using a standard  fixed point result we then prove the local exact  controllability result for the non linear problem,
\begin{equation*}
\leqno(KdV_{control}) 
\left \{
\begin{array}{ll}
(\partial_t u_{j} + \partial_x u_{j} + u_j \partial_x u_j + \partial^3_x u_j)(t,x)=0,&\, \forall \, x \in(0,\ell_j),\, t>0,\, j = 1,...,N, \\
u_{j}(t,0)=u_{k}(t,0),&\, \forall \, j,k = 1,...,N, \, t > 0, \\
\ds \sum_{j=1}^N \partial^2_{x} u_j (t,0) = - \alpha \, u_1(t,0) - \frac{N}{3} u^2_1(t,0)+g(t),& \, \forall \, t > 0, \\
u_j(t,\ell_j) = 0,& \, \forall \, t > 0, \, j= 1,...,N,\\
\partial_x u_j (t,\ell_j) = g_j(t),& \, \forall \, t > 0, \, j= 1,...,N, \\
u_j(0,x)=u_j^0(x),& \, \forall \, x \in (0,\ell_j),\,   j=1,...,N,
\end{array}
\right.
\end{equation*}
\begin{theorem}
Let $T>0$ and $(\ell_i)_{i=1,\ldots,N}\in (0,+\infty)^N$ such that $\# \{\ell_i\in \mathcal N\}\leq 1$. Then there exists $r>0$ such that  for all $\underline u^0,\underline u^T\in \mathbb L^2(\mathcal T)$ with $\norm{\underline u^0}{\mathbb L^2(\mathcal T)}<r$ and $\norm{\underline u^T}{\mathbb L^2(\mathcal T)}<r$, there exists $\underline g\in \mathbb{L}^2(0,T):= \ds \prod_{j=1}^N L^2(0,T)$ and $g\in L^2(0,T)$ such that the solution $\underline u\in \mathbb B$ of $(KdV_{control})$ satisfies $\underline u(0,.)=\underline u^0$ and $\underline u(T,.)=\underline u^T$.
\end{theorem}
\begin{remark}
If $\# \{\ell_i\in \mathcal N\}\leq 1$, there exists a finite dimensional space of $\mathbb L^2(\mathcal T)$ which is unreachable for the linearized system  $(LKdV_{control})$.  We could certainly prove the controllability of the non linear problem by using some power series expansion for the critical branches, following the same type of proof as \cite{Coron_Crepeau}, \cite{Cerpa1} or \cite{Cerpa_Crepeau}. 
\end{remark}

\begin{remark}
In this last section, we prove the controllability by using (N+1) controls, acting at the external nodes and at the central node. It could be interesting to reduce the number of controls.
\end{remark}


\begin{thebibliography}{99}


\bibitem{ammarinicaise} K. Ammari and S. Nicaise, Stabilization of elastic systems by collocated feedback,
{\em Lecture Notes in Mathematics,} 2124, Springer, Cham, 2015.

\bibitem{amregvalmer} K. Ammari, D. Mercier, V. R\'egnier and J. Valein, Spectral analysis and stabilization of a chain of serially connected Euler-Bernoulli beams and strings, {\em Commun. Pure Appl. Anal.,} {\bf 11} (2012), 785--807.

\bibitem{ammari4} K. Ammari and M. Jellouli, Remark in stabilization of tree-shaped networks of
strings, {\em Appl. Maths.,} {\bf 4} (2007), 327-343.  

\bibitem{amjel} K. Ammari and M. Jellouli, Stabilization of star-shaped networks of strings, {\em Diff. Integral. Equations,}
{\bf 17} (2004), 1395-1410. 
\bibitem{Cerpa1}E. Cerpa, Exact controllability of a nonlinear Korteweg-de Vries equation on a critical spatial domain, {\em SIAM Journal on Control and Optimization}, {\bf 46} (2007), 877-899.
\bibitem{Cerpa} E. Cerpa, Control of a Korteweg-de Vries equation: a tutorial, {\em Math. Control Relat. Fields}, {\bf 4} (2014), 45--99.
\bibitem{Cerpa_Crepeau}E. Cerpa and E. Cr\'epeau, Boundary controllability for the nonlinear Korteweg-de Vries equation on any critical domain, {\em Annales de l'Institut Henri Poincare (C) Non Linear Analysis }, {\bf  26} (2009), 457-475. 
\bibitem{Coron_Crepeau}J.M. Coron and E. Cr\'epeau, Exact boundary controllability of a nonlinear KdV equation with critical lengths, {\em Journal of the European Mathematical Society}, {\bf 6} (2004), 367-398.
\bibitem{Crepeau}E. Cr\'epeau, Exact boundary controllability of the Korteweg-de Vries equation with a piecewise constant main coefficient, {\em  Systems $\&$ Control Letters}, {\bf 97} (2016), 157-162.
\bibitem{Crepeau-Sorine} E. Cr\'epeau and  M. Sorine,  A reduced model of pulsatile flow in an arterial compartment,{\em  Chaos, Solitons $\&$ Fractals}, {\bf 34(2)} (2007), 594-605.
\bibitem{dagerzuazua} R. D{\'a}ger and E. Zuazua, {\em Wave propagation, observation and control in 
{$1\text{-}d$} flexible multi-structures}, 50, Math\'ematiques \& Applications (Berlin), Springer-Verlag, 2006.

\bibitem{Lions}J.L. Lions, {\em Contr\^olabilit\'e exacte, stabilisation et perturbations de systèmes distribu\'es, Tome 1}, Contr\^olabilit\'e exacte, Recherches en mathematiques appliquées, 1988.
\bibitem{Pazoto} A.F. Pazoto, A. F., Unique continuation and decay for the Korteweg-de Vries equation with localized damping, {\em  ESAIM: Control, Optimisation and Calculus of Variations},{\bf 11} (2005), 473-486.

\bibitem{Pazy} A. Pazy, {\em Semigroups of linear operators and applications
to partial differential equations}, Springer, New York, 1983.

\bibitem{Perla}G. Perla Menzala, C.F. Vasconcelos and E. Zuazua, Stabilization of the Korteweg-de Vries equation with localized damping, {\em Quarterly of applied Mathematics}, {\bf 60} (2002), 111--129.

\bibitem{Rosier} L. Rosier, Exact boundary controllability of the Korteweg-de Vries equation on a bounded domain, {\em ESAIM: COCV}, {\bf 2} (1997), 33--55.

\bibitem{Saut}J.C. Saut and B. Scheurer, B., Unique continuation for some evolution equations, {\em  Journal of differential equations}, {\bf 66} (1987), 118-139.

\end{thebibliography}
\end{document}